\documentclass{article}

\usepackage{microtype}
\usepackage{graphicx}
\usepackage{subfigure}
\usepackage{tikz}
\usepackage{booktabs} 

\usepackage{hyperref}

\newcommand{\rebuttal}[1]{\textcolor{black}{#1}}

\usepackage{mathtools}

\usepackage[accepted]{icml2024}

\usepackage{amsmath}
\usepackage{amssymb}
\usepackage{mathtools}
\usepackage{amsthm}
\usepackage{thmtools} 
\usepackage{thm-restate}

\usepackage[capitalize,noabbrev]{cleveref}

\theoremstyle{plain}
\newtheorem{theorem}{Theorem}

\newtheorem{lemma}[theorem]{Lemma}

\theoremstyle{definition}

\theoremstyle{remark}
\newtheorem{remark}[theorem]{Remark}

\usepackage[textsize=tiny]{todonotes}

\usepackage{comment}

\icmltitlerunning{
Optimization in SciML Should Employ the Function Space Geometry
}

\begin{document}

\twocolumn[
\icmltitle{Position: Optimization in SciML Should Employ the Function Space Geometry}




\begin{icmlauthorlist}
\icmlauthor{Johannes Müller}{yyy}
\icmlauthor{Marius Zeinhofer}{sch}
\end{icmlauthorlist}

\icmlaffiliation{yyy}{Chair of Mathematics of Information Processing, RWTH Aachen University, Aachen, Germany}
\icmlaffiliation{sch}{Simula Research Laboratory, Oslo, Norway}

\icmlcorrespondingauthor{Johannes M\"uller}{mueller@mathc.rwth-aachen.de}
\icmlcorrespondingauthor{Marius Zeinhofer}{mariusz@simula.no}

\icmlkeywords{Machine Learning, ICML}

\vskip 0.3in
]



\printAffiliationsAndNotice{}  

\begin{abstract}
Scientific machine learning (SciML) is a relatively new field that aims to solve problems from different fields of natural sciences using machine learning tools. 
It is well-documented that the optimizers commonly used in other areas of machine learning perform poorly on many SciML problems.
We provide an infinite-dimensional view on optimization problems encountered in scientific machine learning and advocate for the paradigm \emph{first optimize, then discretize} for their solution. 
This amounts to first choosing an appropriate infinite-dimensional algorithm which is then discretized in a second step. 
To illustrate this point, we show that recently proposed state-of-the-art algorithms for SciML applications can be derived within this framework. 
As the infinite-dimensional viewpoint is presently underdeveloped in scientific machine learning, we formalize it here and advocate for its use in SciML in the development of efficient optimization algorithms.   
\end{abstract}

\section{Introduction}\label{sec:introduction}
The investigation of optimization problems in function spaces dates back at least to Newton's minimal resistance problem~\cite{newton1987philosophiae}, sparking the field today known as the calculus of variations. It took the mathematical community until the beginning of the 20th century for the rigorous concept of function spaces to emerge, initiated by works of Fr\'echet, Hilbert, Riesz, and Fischer among many others. 
Today, optimization problems in function spaces are the basis for many applications in engineering and science, including fluid dynamics, solid mechanics, quantum mechanics, and optimal design to name but a few areas of applications. Optimization in function spaces is intimately connected to the solution of \emph{partial differential equations} (PDEs).

Recently, efforts have been made to apply machine learning methods to scientific problems posed in function spaces. Among the goals are the seamless integration of observational data and (partial) physical knowledge in the form of PDEs using physics-informed neural networks (PINNs), the solution of high-dimensional PDEs, such as the many-electron Schr\"odinger equation, and the design of fast neural surrogate models for many-query applications to replace computationally expensive physics-based models, typically referred to as neural operators. We collect these different approaches under the name Scientific Machine Learning (SciML), see also Subsection~\ref{sec:sciml_methods} for more details. Common to all of the methods is that they can be formulated in infinite-dimensional function spaces that are dictated by the underlying PDEs.

Many of the SciML methods lead to notoriously hard optimization problems in practice, different from purely data-driven deep learning applications. 
For PINNs in particular, it is well documented that first-order methods plateau without reaching high accuracy~\cite{krishnapriyan2021characterizing, wang2021understanding, zeng2022competitive}, which has led many people to believe that 
\begin{quote}
    \emph{there has been a lack of research on optimization tasks for PINNs}~\cite{cuomo2022scientific}. 
\end{quote} 
In particular, for neural network based PDE solvers, the optimization error seems to be dominating the approximation and generalization error, hence rendering these problems very different from many deep learning related tasks.
More principled approaches to optimization in SciML have just started to be developed~\cite{pmlr-v202-muller23b, zeng2022competitive, anonymous2024}. For variational Monte Carlo methods with neural network ansatz functions the gold standard is K-FAC~\cite{hermann2022ab}, however, the optimization process is still very resource-consuming~\cite{pfau2020ab,li2023forward}. In general, there is no clear consensus on best practices for principled choices of optimization algorithms, however, there is a tendency towards second-order optimization algorithms unlike in standard deep learning applications.

\paragraph{Function-space inspired methods for SciML}
Function-space inspired methods provide a principled way to design optimizers that take the specific problem structure into account. 
Such approaches have been proposed and used in supervised learning and reinforcement learning~\cite{amari1998natural, kakade2001natural, martens2020new} to great success. Here, we discuss their potential in the context of SciML. 
\textbf{
We advocate the position that modern optimization algorithms in SciML should employ the function space geometry.}
The idea is to choose a suitable infinite-dimensional optimization algorithm based on the properties of the continuous formulation.
Only after the choice of an optimization algorithm, the problem is discretized. 
This corresponds to the well-known \emph{first optimize, then discretize} approach in PDE-constrained optimization, where linear methods rather than networks are used~\cite{hinze2008optimization}. 
We summarize our main conclusions as follows:
\begin{itemize}
    \item We argue that most state-of-the-art optimizers in SciML can be obtained from a function-space principle, although they are not commonly motivated this way, see Section~\ref{sec:concrete_examples}. 
    This way, we identify an emerging trend in optimization algorithms for SciML. 
    \item Descpite the initial success of function-space inspired methods in SciML, they have not yet been widely adopted in the field. Hence, we advocate for their use and development as we believe this has the potential to help SciML mature into off-the-shelf technology that can be applied at an industrial scale. 
\end{itemize}

\rebuttal{The article is structured as follows. } 
\rebuttal{
In Section~\ref{sec:preliminaries} we discuss the problems of consideration form the field of Scientific machine learning as well as the importance of the development of more principled and efficient optimizers in this field.  
In Section~\ref{sec:optimization_in_hilbert}, we describe the derivation of function-space inspired methods:} Formulate the problem at the continuous level. Typically, one views the problem as an infinite-dimensional problem discretized by a neural network ansatz.
Choose an optimization algorithm that is well-motivated for the infinite-dimensional problem\footnote{The reader may for instance think of Newton's method one-step convergence for the solution of a quadratic problem.} and discretize the algorithm in the tangent space of the neural network ansatz. 
\rebuttal{In Section~\ref{sec:concrete_examples}, we discuss how different state-of-the-art methods in SciML can be derived within this framework and discuss their superiority to first-order and certain second-order methods.} 

\section{\rebuttal{Scientific Machine Learning}\label{sec:preliminaries}}
\rebuttal{We describe the problems from the field of scientific machine learning we consider here. 
Further, we give an overview the current optimization approaches to these problems highlighting that insufficient optimization is widely believed to be the main bottleneck in this field.  
}

\subsection{Methods of Scientific Machine Learning}\label{sec:sciml_methods}
We briefly discuss a sample of influential works in scientific machine learning. 
The unifying theme is the desire to make neural networks conform to physical laws described in terms of partial differential equations. 
In the later sections, we refer to these examples and discuss the importance of infinite-dimensional optimization algorithms for them.

\paragraph{Physics-Informed Neural Networks}
Proposed by~\cite{dissanayake1994neural} and popularized by~\cite{raissi2019physics}, physics-informed neural networks merge observational data with physical prior knowledge in the form of partial differential equations, see also the review articles~\cite{karniadakis2021physics, cuomo2022scientific}. 
A typical example is the task of finding a function $u_\theta$ parametrized as a neural network that matches observational data $u_d$ and satisfies a partial differential equation, say
\begin{align*}
    \partial_tu + \mathcal N(u) &= 0,
\end{align*}
with given boundary and initial data. Here, $\mathcal N$ denotes a partial differential operator. 
The PINN formulation\footnote{We have omitted terms corresponding to initial and boundary values for brevity of presentation.} of the above problem reads
\begin{equation}\label{eq:PINN}
    \min_{\theta\in\Theta}L(\theta)
    = 
    \frac12\| u_\theta - u_d \|^2_{L^2} + \frac12\| \partial_tu_\theta + \mathcal N(u_\theta) \|^2_{L^2}.
\end{equation}
Minimizing $L$ corresponds to solving the PDE and conforming to the observational data $u_d$. 
Alternatively, if a variational principle for the PDE at hand is available, this can be used to formulate a loss function. 
This approach, known as the deep Ritz method, was proposed in~\cite{weinan2018deep}. 
We remark that physics-informed neural networks and the deep Ritz method can also be regarded as a method to solve PDEs, ignoring the data term in the loss above. 

\paragraph{Variational Monte Carlo Methods}
Recently, neural network based variational Monte Carlo (VMC) methods have shown promising results for the solution of the many-electron Schr\"odinger equation and presently allow ab initio solutions for system sizes of the order of $100$ electrons. 
We refer to~\cite{hermann2022ab} for an overview of this fast-growing field. 
The idea is to parametrize the wave function $\psi$ by a neural network $\psi_\theta$, where we denote the trainable parameters by $\theta$ and to minimize the Raleigh quotient 
\begin{equation*}
    \min_{\theta\in\Theta}\frac{\langle \psi_\theta|\hat H |\psi_\theta \rangle}{\langle \psi_\theta|\psi_\theta \rangle}
\end{equation*}
where $\hat H$ is the Hamiltonian, which is typically a linear second-order partial differential operator. 

\paragraph{Operator Learning}
Operator learning aims at approximating an operator $\mathcal G\colon X \to Y$ between function spaces $X$ and $Y$ by a neural surrogate $\mathcal G_\theta \approx \mathcal G$, see~\cite{kovachki2023neural, li2020fourier}. 
A prototypical example is the emulation of the solution map $\mathcal G\colon f \mapsto u$ of a given PDE, where $u=\mathcal Gf$ is the solution of the PDE with data $f$. 
Neural operators are typically trained using a regression formulation, although incorporating PDE information is also possible~\cite{li2021physics}. Combining both, a neural operator loss function has the form
\begin{equation*}
    L(\theta) 
    =
    \frac12 \sum_{i=1}^N \|\mathcal G_\theta (f_i) - u_i\|^{2}_Y
    +
    \|\mathcal D(\mathcal G_\theta (f_i)) - f_i\|^2_Y,
\end{equation*}
where $(f_i, u_i)_{i=1,\dots,N}$ denotes the training data, and the $u_1,\dots,u_N$ is typically generated by a classical solver. 
The second term in the objective may reduce the need for training data at the expense of a more difficult optimization problem. 

\subsection{Optimization in Scientific Machine Learning} 
\rebuttal{
In their nature, PINNs and related methods are very different from the problems encountered in supervised learning and reinforcement learning. 
Indeed, the points used in the numerical discretization of the objective function~\eqref{eq:PINN} play the role of data points, hence one has access to an unlimited amount of data, which renders the problem as an optimization rather than a statistical one. 
}
\rebuttal{
Further, it is known that the optimization problems are very badly conditioned and 
consequently, it is 
commonly believed that optimization is one of the biggest challenges in SciML with training pathologies being well documented~\cite{wang2021understanding, krishnapriyan2021characterizing, cuomo2022scientific, de2023operator, liu2024preconditioning}}. 
\rebuttal{
To address these difficulties in the optimization, various adaptive weightings of the loss~\cite{wang2021understanding, wang2022and} and sampling strategies~\cite{lu2021deepxde, nabian2021efficient, daw2022rethinking, zapf2022investigating, wang2022respecting, wu2023comprehensive, tang2023pinns, jiao2023gas} have been suggested, improving naive methods, but failing to produce satisfactory accuracy. 
}
Achieving highly accurate PINN solutions has just recently been realized in~\cite{zeng2022competitive, pmlr-v202-muller23b}. 
As we show in Section~\ref{sec:concrete_examples}, all methods that succeed in producing accurate solutions can be interpreted from an infinite-dimensional viewpoint. \rebuttal{An illustration of the importance of the infinite-dimensional perspective is provided in Figure~\ref{fig:training_curves}, which demonstrates that respecting the function space geometry can result in orders of magnitude improvement.  }
\begin{figure}
    \centering
    \includegraphics[width=0.5\textwidth, trim={0 0.2cm 0cm 0cm},clip]{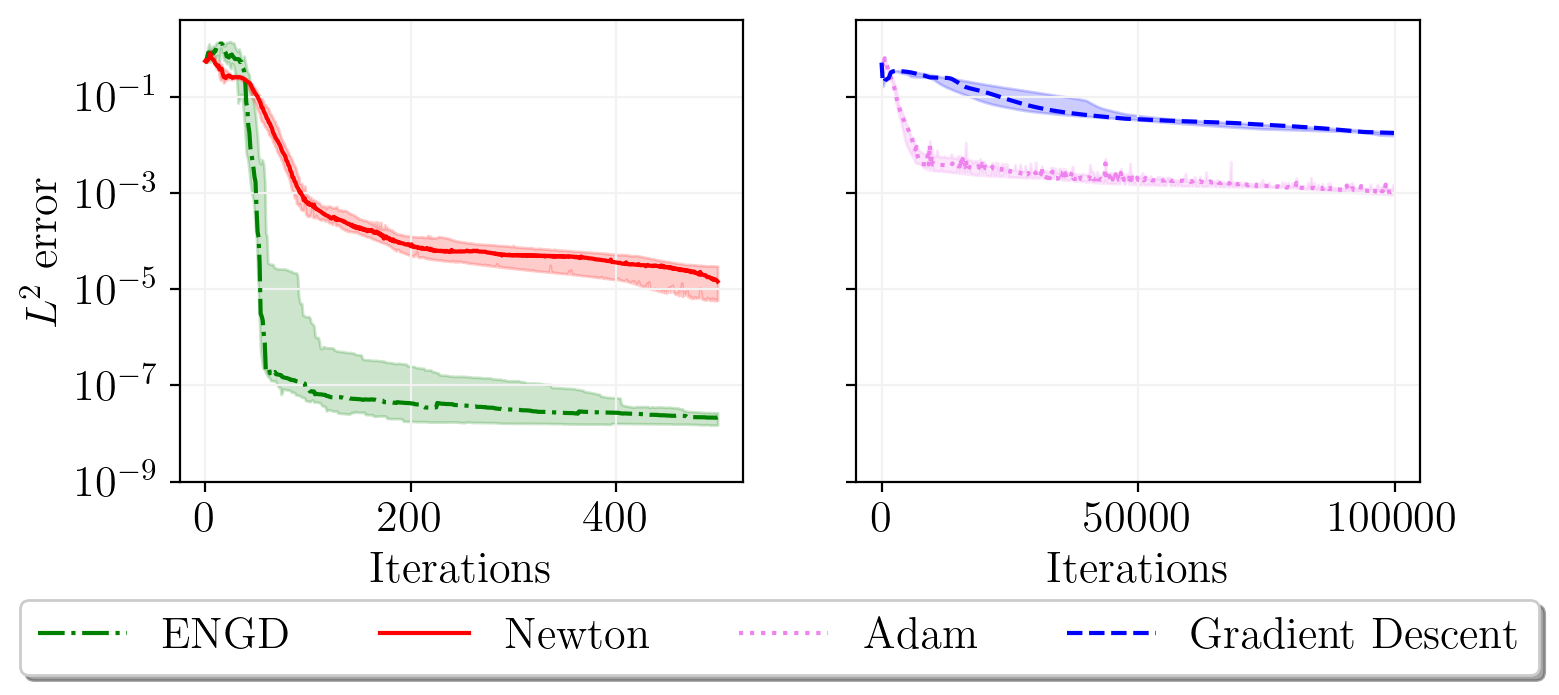}
    \vspace{-.5cm}
    \caption{\rebuttal{Training curves for a PINN for a 2-dimensional Poisson equation; the first order optimizers (Adam and Gradient Descent) plateau, the second-order optimizers (ENGD and Newton) perform much better, but the function-space inspired optimizer (ENGD) reaches the highest accuracy by several orders of magnitude.}}
    \label{fig:training_curves}
\end{figure}

Variational Monte Carlo with neural network ansatz functions routinely relies on the natural gradient method as proposed in~\cite{amari1998natural} and uses the K-FAC approximation of the Fisher matrix, see~\cite{martens2015optimizing, martens2020new}. 
In Section~\ref{sec:l2_vmc}, we illustrate the derivation of the natural gradient method for VMC as an $L^2$-gradient descent algorithm in function space. 
Note that for large problem sizes, applications become dramatically resource-demanding, and computation times up to $10^4$ GPU hours for a single experiment are reported in~\cite{li2023forward}. 
This illustrates the importance of a well-suited training algorithm.

Neural operators for PDE-based applications usually require an offline data generation phase using classical, grid-based solvers to produce training data. 
Consequently, the neural operator needs to be trained before it is eventually ready for downstream applications. 
The training is typically carried out with stochastic first-order methods like Adam.

\section{\rebuttal{Function-Space Inspired Optimization}} \label{sec:optimization_in_hilbert}
In this section, we discuss the general principle of the discretization of iterative algorithms in Hilbert spaces, specific examples are provided in Section~\ref{sec:concrete_examples}. 
Assume that we aim to solve a problem in a Hilbert space $\mathcal H$, say an optimization or saddle-point problem. 
To tackle this problem numerically, we consider neural network functions of a given architecture that form a parametric class of functions 
\begin{equation*}
    \mathcal M = \{ u_\theta : \theta\in\Theta \} \subseteq \mathcal H,
\end{equation*}
where $\Theta=\mathbb R^p$ is the parameter space.  
It is usually straightforward to design neural networks such that $u_\theta\in\mathcal H$ as this typically relies on smoothness properties of the activation function. 
We will sometimes require the map
\begin{equation}\label{eq:parametrization}
    P\colon\Theta \to \mathcal M\subseteq \mathcal H, \quad \theta\mapsto u_\theta,
\end{equation}
which we call the \emph{parametrization} {and assume that it is differentiable}. We furthermore define the \emph{generalized tangent space} of $\mathcal M$ at $u_\theta$ as
\begin{equation}\label{eq:tangent_space}
    \operatorname{span}
    \{
    \partial_{\theta_1}u_\theta,\dots,\partial_{\theta_p}u_\theta
    \}
    =
    \operatorname{ran}(DP(\theta)) \subset\mathcal H,
\end{equation}
where $DP(\theta)$ denotes the derivative of the parametrization. Here, $\partial_{\theta_1}u_\theta, \dots, \partial_{\theta_p}u_\theta$ might be linearly dependent.

The definitions above are not specific for neural networks but are meaningful for general parametric ansatz classes. 

\rebuttal{
The idea of designing optimization algorithms for parametric models that take the function space into account dates back to the seminal works of Amari, who provided a discretization of a function space gradient descent known as natural gradient descent~\cite{amari1998natural} and we refer to the surveys of~\cite{ollivier2017information, martens2020new} on function space inspired methods utilizing the Fisher-Rao geometry. 
Apart from their use in supervised learning, where it is attributed to be more robust to data noise~\cite{amari2020does}, natural gradient methods enjoy great popularity in reinforcement learning~\cite{kakade2001natural, peters2003reinforcement, morimura2008new, moskovitz2020efficient}. 
Note that the methods developed in supervised and reinforcement learning can not directly be applied in SciML, where the function spaces often consists of deterministic functions rather then probabilistic models. 
Here, we describe the general philosophy of function-space inspired methods. }

\subsection{Discretization of Minimization Methods}
\rebuttal{Consider an uncontrained optimization problem}
\begin{equation}\label{eq:abstract_minimization}
    \min_{u\in\mathcal H}E(u),
\end{equation}
where $E\colon\mathcal H\to\mathbb R$ is a differentiable function on a Hilbert space $\mathcal H$.
For this we consider the {following} scheme 
\begin{equation}\label{eq:abstract_minimization_algorithm}
    u_{k+1} = u_k + \eta_k d_k, 
\end{equation}
with an update direction $d_k$ that is given by 
\begin{equation}\label{eq:gradient_based_update}
    d_k = -T_{u_k}^{-1}(DE(u_k)),
\end{equation}
where $T_{u_k}\colon\mathcal H\to\mathcal H^*$ is an invertible linear map that is given by the concrete algorithm at hand. For example, we recover gradient descent with $d_k=-\nabla E(u_k)$ and Newton's method with $d_k = -D^2E(u_k)^{-1}(DE(u_k))$. In Section~\ref{sec:concrete_examples}, we provide explicit examples and discuss different choices of function space algorithms in more detail.  

Corresponding to~\eqref{eq:abstract_minimization}, we define the loss function
\begin{equation*}
    L\colon\Theta \to \mathbb R, \quad L(\theta) = E(u_\theta)
\end{equation*}
and aim to design an algorithm in parameter space 
\begin{equation}\label{eq:parameter_algorithm}
    \theta_{k+1} = \theta_k + \eta_k w_k, 
\end{equation}
such that $u_{\theta_k}\approx u_k$. 
To understand the dynamics we apply Taylor's theorem to  $u_{\theta_{k+1}}=P(\theta_k + \eta_k w_k)$ and obtain 
\begin{equation*}
    u_{\theta_{k+1}} = u_{\theta_k} + \eta_k DP(\theta_k) w_k + O(\eta_k^2 \lVert w_k \rVert^2). 
\end{equation*}
To mimic~\eqref{eq:abstract_minimization_algorithm}, we would like that $DP(\theta_k)w_k\approx d_k$ for which we choose 
\begin{equation}\label{eq:fitting_the_update_direction}
    w_k\in\arg\min_{w\in\mathbb R^p} \frac12\lVert DP(\theta_k)w- d_k \rVert_{T_{u_{\theta_k}}}^2,
\end{equation}
where $\| w \|^2_{T_{u_{\theta_k}}} = \langle T_{u_{\theta_k}}w,w \rangle$. Computing the normal equations of~\eqref{eq:fitting_the_update_direction} yields that the update direction is given by
\begin{equation}\label{eq:gramian_update_direction}
    w_k = -G(\theta_k)^\dagger \nabla L(\theta_k),
\end{equation}
where $G(\theta_k)\in\mathbb R^{p\times p}$ denotes the Gramian matrix
\begin{equation}
    G(\theta_k)_{ij}
    = 
    \langle T_{u_\theta} \partial_{\theta_i} u_\theta, \partial_{\theta_j} u_\theta\rangle,
\end{equation}

see Appendix~\ref{appendix:proof_best_fit}. 
By $G(\theta_k)^\dagger$ we denote a pseudo-inverse of $G(\theta_k)$, i.e., a matrix satisfying $GG^\dagger G = G$.
The discretized algorithm now reads
\begin{equation}\label{eq:discretized_algorithm}
    \theta_{k+1} = \theta_k - \eta_k G(\theta_k)^\dagger \nabla L(\theta_k).
\end{equation}
A typical damping strategy consists of adding an $\epsilon_k$-scaled identity to $G$.
We have derived the discretized algorithm by fitting the update direction~\eqref{eq:fitting_the_update_direction}. This implies that the update direction $DP(\theta_k)w_k$ in function space is the projection of $d_k$ onto $\operatorname{ran}(DP(\theta_k))$, i.e., the tangent space of the model, which is well known for natural gradients in the finite-dimensional setting, see \cite{amari2016information,van2022invariance}, where we defer the proof to Appendix~\ref{sec:appendix_update_direction}. 

\begin{restatable}{theorem}{projectionthm}
\label{thm:projection_theorem}
Assume we are in the above setting, i.e., consider an algorithm of the form~\eqref{eq:abstract_minimization_algorithm} that satisfies~\eqref{eq:gradient_based_update}. We assume additionally that the $T_{u_{\theta_k}}$ are symmetric and positive definite. Then, for the discretized algorithm~\eqref{eq:discretized_algorithm} it holds
    \begin{align}\label{eq:FS_dynamics_discretized_algorithm}
        u_{\theta_{k+1}} 
        =  
        u_{\theta_k} 
        -
        \eta_k \Pi_{u_{\theta_k}}[
        T_{u_{\theta_k}}^{-1}(DE(u_{\theta_k}))
        ] + \epsilon_k,
    \end{align}
    where $\Pi_{u}$ denotes the orthogonal projection onto the tangent space with respect to the inner product $\langle T_{u}\cdot, \cdot \rangle$. The term $\epsilon_k$ corresponds to an error vanishing quadratically in the step and step size length 
    \begin{equation*}
        \epsilon_k = O(\eta_k^2\lVert G(\theta_k)^\dagger \nabla L(\theta_k) \rVert^2).
    \end{equation*}
\end{restatable}

Inspecting the function space dynamics~\eqref{eq:FS_dynamics_discretized_algorithm} of the discretized algorithm, we see that they agree with the original function space algorithm~\eqref{eq:abstract_minimization_algorithm} up to the orthogonal projection onto the tangent space of the model and an error vanishing quadratically with the step size. 
\rebuttal{See also Figure~\ref{fig:pushes}, which confirms that the function-space inspired methods (E-NG and GN-NG) lead to function updates that compensate the error of the method much better. }

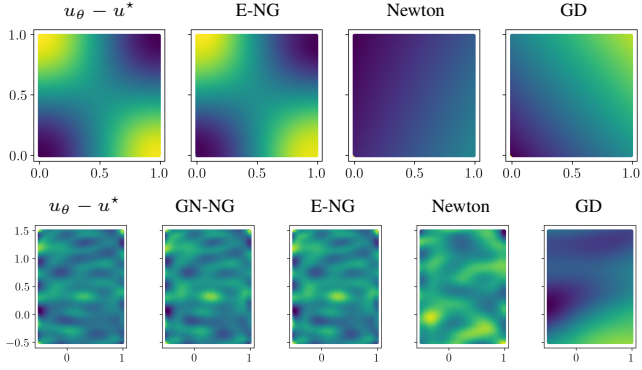
\begin{figure}
    \centering
    \begin{tikzpicture}
        \node[inner sep=0pt] (r1) at (0,0)
    {\includegraphics[width=0.5\textwidth, trim={0 0 2.7cm .9cm},clip]{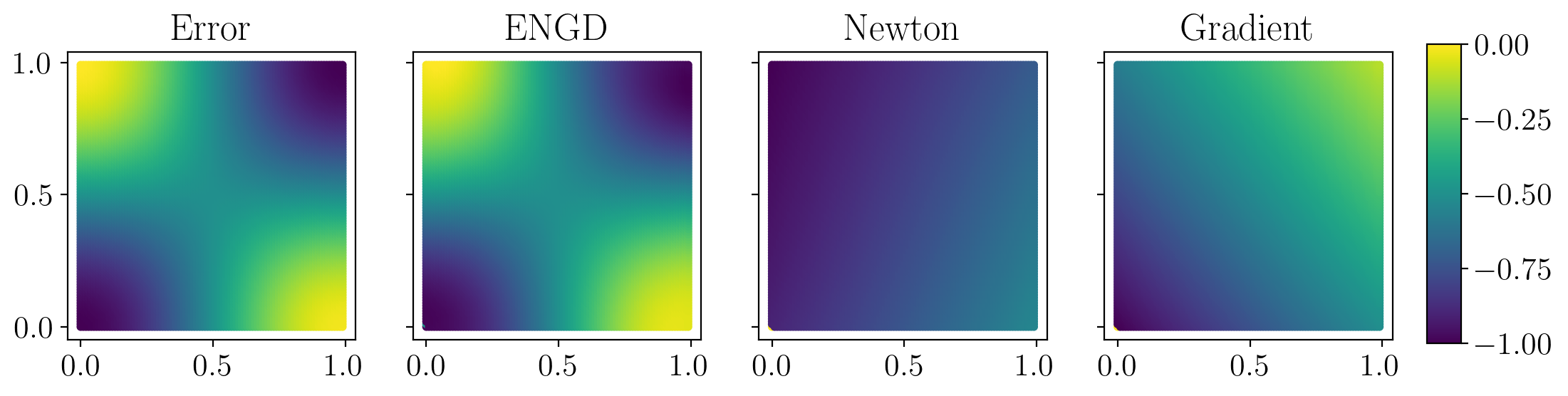}};
        \node[inner sep=0pt] (r1) at (0,-2.5)
    {\includegraphics[width=0.5\textwidth, trim={0.1cm 0 2.9cm 0.8cm},clip]{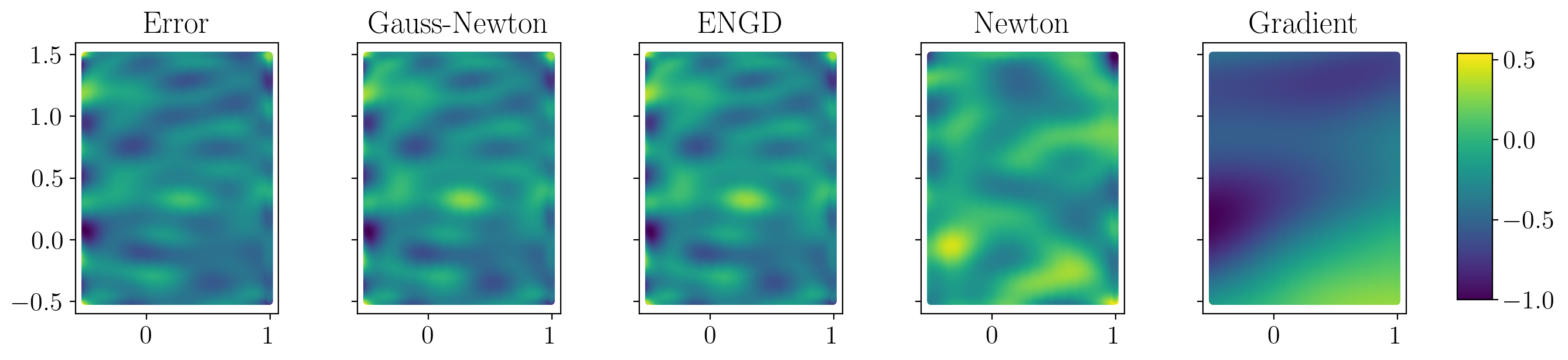}};
        \node[inner sep=0pt] (r1) at (-3,1.3)
    {\scriptsize $u_\theta-u^\star$};
        \node[inner sep=0pt] (r1) at (-.9,1.3)
    {\scriptsize E-NG};
        \node[inner sep=0pt] (r1) at (1.2,1.3)
    {\scriptsize Newton};
        \node[inner sep=0pt] (r1) at (3.3,1.3)
    {\scriptsize GD};
        \node[inner sep=0pt] (r1) at (-3.2,-1.3)
    {\scriptsize $u_\theta-u^\star$};
        \node[inner sep=0pt] (r1) at (-1.6,-1.3)
    {\scriptsize GN-NG};
        \node[inner sep=0pt] (r1) at (0.1,-1.3)
    {\scriptsize E-NG};
        \node[inner sep=0pt] (r1) at (1.8,-1.3)
    {\scriptsize Newton};
        \node[inner sep=0pt] (r1) at (3.5,-1.3)
    {\scriptsize GD};
    \end{tikzpicture}
    \vspace{-.5cm}
    \caption{\rebuttal{Shown is the heat map over the domain $\Omega$ of the function space updates for two different problems, a linear Poisson equation (top) and a steady-state Navier-Stokes system (bottom), for different optimizers, being an energy natural gradient (E-NG), a Gauß-Newton natural gradient (GN-NG), which agrees with E-NG for the Poisson equation, Newton's methods and gradient descent (GD); In addition, the error $u_\theta-u^\star$ is shown, where $u^\star$ is the true solutions; note that the functions-space inspired methods (E-NG and GN-NG) lead to updates that can compensate the error better.}}
    \label{fig:pushes}
\end{figure}

\begin{remark}[Connection to Galerkin Discretization]
Galerkin schemes in numerical analysis refer to the discretization of linear forms and linear maps using finite dimensional vector spaces~\cite{brenner2008mathematical}. 
We can interpret the discretization procedure discussed above as a Galerkin scheme in the models tangent space, i.e., using the basis functions $\partial_{\theta_1}u_\theta,\dots, \partial_{\theta_p}u_\theta$. 
Discretizing the Fr\'echet derivative $DE(u_\theta)$ and $T_{u_\theta}$ this way yields $\nabla L(\theta)$ and $G(\theta)$. 
However, unlike in classical methods, the space used in the Galerkin discretization changes at every iteration. 
\end{remark}

\begin{remark}[Extension to Non-Symmetric \& Indefinite $T_u$]
If $T_u$ is symmetric and positive definite, it naturally defines a Riemannian metric. In this case, the discretized algorithm in~\eqref{eq:discretized_algorithm} agrees with a natural gradient method. When $T_{u}$ is non-symmetric or indefinite,~\eqref{eq:fitting_the_update_direction} is not the right notion to obtain an update direction, as $T_u$ does not correspond to an orthogonal projection in the strict sense. In this case, we still use~\eqref{eq:gramian_update_direction} although it is not the optimality condition of~\eqref{eq:fitting_the_update_direction}. Assuming, for instance, coercivity or suitable inf-sup conditions of $T_u$ guarantees that $w_k$ is a quasi-best approximation of $d_k$. For the details, we refer to Appendix~\ref{sec:extension_to_non_symmetric}.
\end{remark}

\subsection{Discretization of Saddle Point Problems}
Another important problem class is given by saddle point problems. Assume we are given a function 
\begin{equation*}
    \mathcal L\colon \mathcal H \times \mathcal V \to \mathbb R
\end{equation*}
and we are looking for a saddle point $(u^*,v^*)$, which means 
\begin{equation*}
    \mathcal L(u^*, v) \leq \mathcal L(u^*, v^*) \leq \mathcal L(u, v^*)
\end{equation*}
is satisfied for all $u\in\mathcal H$ and $v\in\mathcal V$. Such a point solves a minimax problem of the form
\begin{equation*}
    \min_{u\in\mathcal H}\max_{v\in\mathcal V}\mathcal L(u,v).
\end{equation*}
Saddle point problems arise, for instance, from Lagrangian approaches to constrained minimization problems. For their solution, we again consider algorithms of the form
\begin{equation*}
    (u_{k+1}, v_{k+1}) = (u_{k}, v_{k}) + (d^u_{k}, d^v_{k})
\end{equation*}
where we assume that $\mathcal L$ is Fr\'echet differentiable and the update direction is given by 
\begin{equation*}
    (d^{u}_{k}, d^v_{k}) = T_{u_k,v_k}^{-1}(D_u\mathcal L(u_k,v_k), D_v\mathcal L(u_k,v_k)). 
\end{equation*}
Here, $T_{u,v}\colon\mathcal H\times\mathcal V \to \mathcal H^*\times\mathcal V^*$ is a bounded, linear, and invertible map.
Algorithms of the form described above include Newton's method for the solution of a critical point of $\mathcal L$, gradient descent-ascent,  competitive gradient descent~\cite{schafer2019competitive} as well as  
\rebuttal{natural hidden gradients \cite{mladenovic2021generalized}}. 

We discretize the problem with two neural networks $u_\theta$ and $v_\psi$ with parameter spaces $\Theta$ and $\Psi$, respectively. This yields
\begin{equation*}
    L\colon\Theta \times \Psi \to \mathbb R, \quad L(\theta, \psi) = \mathcal L(u_\theta, v_\psi)
\end{equation*}
and the parameter update 
\begin{align*}
(\theta_{k+1}, \psi_{k+1}) = (\theta_{k}, \psi_{k})  + \eta_k (w^{u}_k, w^{v}_k),
\end{align*}
the update directions are computed according to 
\begin{align}
    \begin{pmatrix}
        w^u_k \\ w^v_k
    \end{pmatrix} = G(\theta_k, \psi_k)^\dagger 
    \begin{pmatrix}
        \nabla_\theta L(\theta_k, \psi_k), \\ \nabla_\psi L(\theta_k, \psi_k)
    \end{pmatrix}.
\end{align}
Here, the Gramian carries a block structure and is obtained by using
\begin{equation*}
    (\partial_{\theta_1}u_\theta,0), \dots, (\partial_{\theta_{p_\Theta}}u_\theta,0), (0, \partial_{\psi_1}v_\psi),\dots,(0,\partial_{\psi_{p_\Psi}}v_\psi)
\end{equation*}
as a generating system for the tangent spaces, for details see Appendix~\ref{appendix:saddle_point_discretization}.
For projection results along the lines of Theorem~\ref{thm:projection_theorem} see Appendix~\ref{sec:extension_to_non_symmetric}

\section{Specific Examples of Algorithms}\label{sec:concrete_examples}
We discuss several infinite-dimensional algorithms, illustrating the abstract framework of the previous section. 
As examples for applications, we discuss (i) neural network ansatz classes for the solution of the many-electron Schr\"odinger equation using an $L^2$ gradient descent, (ii) the solution of a nonlinear variational problem using the deep Ritz formulation and Newton's method in function space, (iii) Lagrange-Newton for the solution of a saddle-point formulation of a Poisson equation, and (iv) Gauss-Newton for the PINN formulation of the Navier-Stokes equations. 
In all cases, the infinite-dimensional algorithms translate to state-of-the-art methods. 
Some of these methods were previously proposed, but lacking the infinite dimensional viewpoint. 

\subsection{Hilbert Space Gradient Descent}\label{sec:l2_vmc}
We show that gradient descent in a Hilbert space corresponds to natural gradient descent in parameter space, with the Riemannian metric induced by the inner product of the Hilbert space. 
We exemplify this using the variational formulation of Schr\"odinger's equation. 
We find that the Hilbert space gradient descent discretized in tangent space corresponds to the state-of-the-art optimization method used in quantum variational Monte Carlo methods, where the wavefunction is parametrized as a neural network, see for instance~\cite{hermann2020deep, pfau2020ab}.

\paragraph{Problem Formulation}
We are interested in solving the problem 
\begin{equation}\label{eq:schrodinger_physics_notation}
    \min_{\psi}E(\psi)
    =
    \frac{\langle \psi|\hat H |\psi \rangle}{\langle \psi|\psi \rangle}
\end{equation}
where $\hat H$ is the Hamiltonian, typically a linear second-order partial differential operator, and $\psi$ denotes the wave function, see for instance~\cite{toulouse2016introduction}. 
Using an $L^2(\Omega)$ gradient descent with initial value $u_0$ for the minimization of $E$ amounts to 
\begin{equation}\label{eq:L2_gradient_flow}
    \psi_{k+1} = \psi_k - \eta_k \mathcal I^{-1}(DE(\psi_k)), \quad k=0,1,2,\dots
\end{equation}
where $\eta_k>0$ denotes a step size and $\mathcal I\colon L^2(\Omega)\to L^2(\Omega)^*$ is the Riesz isometry of $L^2(\Omega)$, given by $\psi\mapsto \langle\psi|\cdot\rangle$. Note that~\eqref{eq:L2_gradient_flow} is formal\footnote{For the discretized algorithm this is not a problem.} as we can not in general guarantee that $DE(\psi_k)\in L^2(\Omega)^*$.

\paragraph{Neural Network Discretization}
We choose a neural network ansatz $\psi_\theta$ for the wavefunction and define the loss
\begin{equation*}
    L(\theta)
    =
    \frac{\langle \psi_\theta|\hat H |\psi_\theta \rangle}{\langle \psi_\theta|\psi_\theta \rangle},
\end{equation*}
where in practice a reformulation is used and the integrals are computed using Monte Carlo quadrature. 
\rebuttal{We use the shorthand notation $\hat\psi = \frac{\psi}{\lVert \psi \rVert_{L^2}}$.}
To discretize~\eqref{eq:L2_gradient_flow}, note that the Riesz isometry $\mathcal I$ corresponds to the map $T_{u_\theta}$ and leads to the Fisher matrix
\begin{equation*}
    G(\theta)_{ij} 
    =
    \mathcal I (\partial_{\theta_i}\hat\psi_\theta)( \partial_{\theta_j}\hat\psi_\theta)
    =
    \int \partial_{\theta_i}\hat\psi_\theta\partial_{\theta_j}\hat\psi_\theta\,\mathrm dx.
\end{equation*}
The update in parameter space thus becomes
\begin{equation*}
    \theta_{k+1} = \theta_k - \eta_k G(\theta_k)^\dagger \nabla L(\theta_k).
\end{equation*}

\paragraph{Correspondence to Natural Gradient Descent}
The above discussion shows that the $L^2(\Omega)$ gradient descent leads to the well-known natural gradient descent using the Fisher information matrix. This algorithm is state of the art for quantum variational Monte Carlo methods with neural network discretization~\cite{pfau2020ab, li2023forward}. Note, that in this context the scalability of the method relies on the K-FAC approximation of $G$. We discuss scalability issues in more detail in Section~\ref{sec:scalability}.

\subsection{Newton's Method}\label{sec:newton_f_space}
Next, we showcase Newton's method for the solution of a semilinear elliptic equation in variational form, i.e., we use the Deep Ritz method~\cite{weinan2018deep} for its solution. 
The example is taken from~\cite{pmlr-v202-muller23b}, where it is shown that this approach yields highly accurate solutions, for both PINN type objectives of linear PDEs and convex minimization problems like the example discussed here.

\paragraph{Problem Formulation} 
Let $\Omega\subset\mathbb R^d$ with $d=1,2,3$ and consider the minimization problem\footnote{Dimensions larger than 3 can be considered as well, but require a different functional setting as in four or more dimensions members of $H^1(\Omega)$ are not integrable in fourth power.}
\begin{equation}\label{eq:semilinear_deep_ritz}
    \min_{u\in H^1(\Omega)}E(u) 
    =
    \int_\Omega \frac{|\nabla u|^2}{2} + \frac{u^4}{4} - f u\,\mathrm dx.
\end{equation}
Newton's method for the minimization of~\eqref{eq:semilinear_deep_ritz} for a given initial value $u_0\in H^1(\Omega)$ is
\begin{equation*}
    u_{k+1} = u_k - D^2E(u_k)^{-1}(DE(u_k)), \quad k=0,1,2,\dots
\end{equation*}
where the Hessian is given by
\begin{equation*}
    D^2E(u)(v,w) = (\nabla v,\nabla w)_\Omega + 3(u^2v, w)_\Omega.
\end{equation*}

\paragraph{Neural Network Discretization}
We choose a neural network ansatz $u_\theta$ with parameter space $\Theta$ and define the loss function
\begin{equation*}
    L(\theta) = E(u_\theta) = \int_\Omega \frac{|\nabla u_\theta|^2}{2} + \frac{ u_\theta^4}{4} - f u_\theta\,\mathrm dx.
\end{equation*}
For the discretization we note that $T_{u_\theta} = D^2E(u_\theta)$, which yields
\begin{align*}
    G(\theta)_{ij} 
    &=
    D^2E(u_\theta)(\partial_{\theta_i}u_\theta,\partial_{\theta_i}u_\theta)
    \\
    &=
    (\nabla \partial_{\theta_i}u_\theta, \nabla\partial_{\theta_j}u_\theta)_\Omega
    +
    3(u_\theta^2 \, \partial_{\theta_i}u_\theta, \partial_{\theta_j}u_\theta)_{\Omega}.
\end{align*}
The algorithm in parameter space becomes 
\begin{equation}\label{eq:engd}
    \theta_{k+1} = \theta_k - \eta_k G(\theta_k)^\dagger \nabla L(\theta_k), \quad k=0,1,2,\dots
\end{equation}
for some initial parameters $\theta_0$ and a stepsize $\eta_k>0$. 

\paragraph{Correspondence to Generalized Gauss-Newton}
To see the connection of the above method to a generalized Gauss-Newton approach \emph{in parameter space}, we compute the Hessian of $L$. In fact, $D^2L(\theta)_{ij}$ equals
\begin{equation*}
    D^2E(u_\theta)(\partial_{\theta_i}u_\theta, \partial_{\theta_j}u_\theta) + DE(u_\theta)(\partial_{\theta_i}\partial_{\theta_j}u_\theta).
\end{equation*}
Generalized Gauss-Newton methods use the first term in the above expansion in an algorithm of the form~\eqref{eq:engd}, which shows that the discretized Newton method and generalized Gauss-Newton coincide, which was recently proposed for the deep Ritz method by~\cite{hao2023gauss}.

\begin{remark}
    The method described above was proposed as a natural gradient method in~\cite{pmlr-v202-muller23b} under the name \emph{energy natural gradient descent}. It was shown that it is highly efficient and capable of producing accurate solutions for PINN-type formulations of linear PDEs and deep Ritz formulations of semilinear elliptic PDEs.
\end{remark}

\begin{remark}\label{remark:engd_is_gn_sometimes}
    Newton in function space for linear PDEs using a PINN formulation yields the standard Gauss-Newton algorithm for $l^2$ regression in parameter space. This is a consequence of the discussion in Section~\ref{sec:gauss_newton_f_space}. We refer to the Remark~\ref{remark:newton_is_gauss_newton_linear_pde} for details.
\end{remark}

\subsection{Lagrange-Newton}\label{sec:lagrange_newton}
The Lagrange Newton algorithm is a solution method for equality-constrained problems~\cite{hinze2008optimization}. It applies Newton's method to find a critical point of the Lagrangian formulation of the constrained problem. We exemplify its application for the solution of a Poisson equation which yields the recently proposed competitive physics-informed neural networks (CPINNs) formulation~\cite{zeng2022competitive}. Moreover, we demonstrate that for linear PDEs the discretization of the Lagrange-Newton algorithm leads to Competitive Gradient Descent (CGD), see~\cite{schafer2019competitive}. The work~\cite{zeng2022competitive} demonstrates in great detail that the application of CGD to the saddle point formulation of linear PDEs yields state-of-the-art accuracy for PINN type optimization problems.

\paragraph{Problem Formulation}
Given $\Omega\subset\mathbb R^d$, $f\in L^2(\Omega)$, $g\in H^{3/2}(\partial\Omega)$ we consider the constant energy $J(u)=0$ and the following problem 
\begin{align}\label{eq:constrained_poisson}
\begin{split}
    \min_{u\in H^2(\Omega)} J(u) \quad \text{s.t. }
    \begin{cases}
        \Delta u + f & = 0 \quad \text{in }\Omega,
    \\
    u - g & = 0 \quad \text{on }\partial\Omega.
    \end{cases}
\end{split} 
\end{align}
Note that the only feasible point of the above minimization problem is the solution $u^*$ of Poisson's equation $-\Delta u^* = f$ with boundary data $g$. The corresponding Lagrangian functional $\mathcal L\colon H^2(\Omega)\times L^2(\Omega) \times L^2(\partial\Omega) \to \mathbb R$ is given by
\begin{equation}\label{eq:lagrangian_cpinns}
    \mathcal L(u,\lambda, \mu) 
    =
    (\lambda , f + \Delta u)_\Omega
    +
    (\mu, u - g)_{\partial\Omega}.
\end{equation}
The unique saddle point or Nash equilibrium of $\mathcal L$ is $(u^\ast, 0, 0)$ and satisfies
\begin{equation*}
    \max_{\lambda, \mu}\min_{u}
    \Big[ 
    (\lambda, f + \Delta u)_\Omega 
    +
    (\mu, u-g)_{\partial\Omega}
    \Big], 
\end{equation*}
see Appendix~\ref{sec:appendix_proofs_lagrange_newton} for details. 
The minimax formulation above is the CPINN formulation proposed in~\cite{zeng2022competitive} applied to Poisson's equation.

Employing the Lagrange-Newton algorithm, we aim to solve
\begin{equation*}
    D\mathcal L(u, \lambda, \mu) = 0
\end{equation*}
via Newton's method. The unique zero of $D\mathcal L$ is $(u^\ast, 0, 0)$ which corresponds to the sought-after saddle point. As $D\mathcal L$ is linear, Newton's method converges in one step for any initial value $(u_0, \lambda_0, \mu_0)$, hence it holds that 
\begin{equation*}
    (u^*, 0, 0)
    =
    (u_0, \lambda_0, \mu_0)
    +
    (d_{u_0}, d_{\lambda_0}, d_{\Delta \mu_0})
\end{equation*}
where the update direction $(d_{u_0}, d_{\lambda_0}, d_{\mu_0})$ is given by
\begin{equation*}
    -D^2\mathcal L(u_0,\lambda_0, \mu_0)^{-1}[D\mathcal L(u_0,\lambda_0, \mu_0)].
\end{equation*}
Further, for test functions $(\delta_u, \delta_\lambda, \delta_\mu)$ and $(\bar\delta_u, \bar\delta_\lambda, \bar\delta_\mu)$ the Hessian  
\begin{equation*}
    D^2\mathcal L(u_0,\lambda_0, \mu_0)((\delta_u, \delta_\lambda, \delta_\mu),(\bar\delta_u, \bar\delta_\lambda, \bar\delta_\mu))
\end{equation*}
is given via the formula
\begin{align}\label{eq:LN-matrix}
    \begin{pmatrix}
        \delta_u \\ \delta_\lambda \\ \delta_\mu
    \end{pmatrix}^T
    \begin{pmatrix}
         0 & (\cdot, \Delta \cdot)_\Omega & (\cdot, \cdot)_{\partial\Omega}
         \\
         (\cdot, \Delta\cdot)_{\Omega} & 0 & 0
         \\
         (\cdot, \cdot)_{\partial\Omega} & 0 & 0
    \end{pmatrix}
    \begin{pmatrix}
        \bar\delta_u
        \\
        \bar\delta_\lambda
        \\
        \bar\delta_\mu
    \end{pmatrix}.
\end{align}

\paragraph{Neural Network Discretization}
We choose neural networks $u_\theta$, $\lambda_\psi$, $\mu_\xi$ and set as a competitive loss 
\begin{equation*}
    L(\theta, \psi, \xi)
    =
    (\lambda_\psi, \Delta u_\theta + f)_\Omega
    +
    (\mu_\xi, u_\theta - g)_{\partial\Omega}.
\end{equation*}
For given initial parameters $\theta_0, \psi_0, \xi_0$ we discretize the Lagrange-Newton algorithm above in tangent space via inserting $\partial_{\theta_i}u_\theta, \partial_{\psi_i}\lambda_\psi, \partial_{\xi_i}\mu_\xi$ into the Hessian $D^2\mathcal L$ and the derivative $D\mathcal L$, see Appendix~\ref{appendix:saddle_point_discretization} for details. This yields a block matrix $G = G(\theta, \psi, \xi)$ of the form 
\begin{align*}
    G(\theta, \psi, \xi) = 
    \begin{pmatrix}
        0   & A & B \\
        A^T & 0 & 0 \\
        B^T & 0 & 0
    \end{pmatrix},
\end{align*}
where
\begin{equation*}
    A_{ij} = (\partial_{\psi_j}\lambda_\psi, \Delta \partial_{\theta_i}u_\theta)_\Omega, 
    \quad
    B_{ij} = (\partial_{\xi_j}\mu_\xi, \partial_{\theta_i}u_\theta).
\end{equation*}
The algorithm becomes
\begin{equation}
    (\theta, \psi, \xi)_{k+1}
    =
    (\theta, \psi, \xi)_k
    -
    \eta_k
    G_k^\dagger \nabla L(\theta_k,\psi_k, \xi_k),
\end{equation}
where $\eta_k$ is a suitably chosen stepsize where  
damping yields 
\begin{equation*}
    (\theta, \psi, \xi)_{k+1}
    =
    (\theta, \psi, \xi)_k
    -
    \eta_k
    (G_k + \epsilon_k\operatorname{Id})^\dagger \nabla L(\theta_k,\psi_k, \xi_k).
\end{equation*}

\paragraph{Correspondence to Competitive GD}
Adapting the notation, we translate equation (4) of~\cite{schafer2019competitive} to our setting, which describes an iteration of CGD. For a fixed small $\eta>0$, the update direction $(d_\theta, d_\psi, d_\xi)$ is given by
\begin{align*}
    -\begin{pmatrix}
    \operatorname{Id} & \eta^{-1} D^2_{\theta,\psi}L & \eta^{-1} D^2_{\theta,\xi}L 
    \\
    \eta^{-1} D^2_{\psi\theta}L & \operatorname{Id} & \eta^{-1} D^2_{\psi,\xi}L
    \\
    \eta^{-1} D^2_{\xi,\theta}L & \eta^{-1} D^2_{\xi,\psi}L & \operatorname{Id}
    \end{pmatrix}^{-1}
    \begin{pmatrix}
        \nabla_\theta L
        \\
        \nabla_\psi L
        \\
        \nabla_\xi L
    \end{pmatrix}
\end{align*}
and this matrix coincides precisely with $\eta^{-1} G + \operatorname{Id}$, for which we provide the elementary computations in the appendix~\ref{sec:appendix_proofs_lagrange_newton}. In compact notation, this corresponds to 
\begin{align*}
    (d_\theta, d_\psi, d_\xi)
    &=
    (\eta^{-1} G + \operatorname{Id})^{-1}DL
    \\
    &=
    \eta [G + \eta\operatorname{Id}]^{-1}DL.
\end{align*}
This means it is precisely a discretized Lagrange-Newton method with damping. 

\begin{remark} The connection between Competitive Gradient Descent on parameter space and Lagrange-Newton in function space can be extended to general \emph{linear} PDEs. However, it does not extend to the nonlinear case, which leads to a non-vanishing first diagonal entry in the matrix~\eqref{eq:LN-matrix}. 
\end{remark}

\subsection{Gauss-Newton}\label{sec:gauss_newton_f_space}
We discuss a function space version of the Gauss-Newton algorithm for the solution of nonlinear least-squares problems. 
We exemplify the algorithm using a PINN-type formulation for the Navier-Stokes equations. 
We show that Gauss-Newton in function space leads to Gauss-Newton in parameter space. 
This approach was proposed in~\cite{anonymous2024} for the solution of the Navier-Stokes equations with neural network ansatz. 
There it is shown that it yields state of the art accuracy for neural network approximations of solutions of the Navier-Stokes equations. 

\paragraph{Problem Formulation} For $\Omega\subset\mathbb R^d$ and forcing $f$, the steady-state Navier-Stokes equations in velocity-pressure formulation are given by
\begin{align}
    \begin{split}\label{eq:navier_stokes}
        -\Delta u + (u\cdot\nabla)u + \nabla p &= f \quad \text{in }\Omega
        \\
        \operatorname{div}u &= 0 \quad \text{in }\Omega,
    \end{split}
\end{align}
with suitable boundary conditions. We reformulate~\eqref{eq:navier_stokes} in least-squares, i.e., PINN form and assume -- for brevity of presentation mainly -- that the ansatz space for the velocity consists of divergence-free functions that satisfy the desired boundary conditions.\footnote{Modifying neural network ansatz spaces to fulfill these requirements is not uncommon, see~\cite{sukumar2022exact, richter2022neural}} This yields
\begin{equation}\label{eq:pinn_navier_stokes}
    E(u,p) = 
    \frac12 \| R(u,p) \|^2_{L^2(\Omega)^d},
\end{equation} 
for a suitably defined nonlinear residual $R$.
Note that the problem is neither quadratic nor convex in $(u,p)$. 
Here, we choose Gauß-Newton as a function space algorithm as we are facing a least squares problem. 
The Gauss-Newton algorithm in function space linearizes $R$ at the current iterate and explicitly solves the resulting quadratic problem, see also~\cite{dennis1996numerical}. Following this strategy, we obtain
\begin{equation*}
    (u_{k+1}, p_{k+1}) = (u_{k}, p_{k}) - T_k^{-1}[DE(u_k, p_k)],
\end{equation*}
where $T_k$ is given by\footnote{More precisely this means that $\langle T_k(\delta_u, \delta_p), (\bar\delta_u, \bar\delta_p) \rangle$ equals $(DR(u_k, p_k)(\delta_u, \delta_p), DR(u_k, p_k)(\bar\delta_u, \bar\delta_p))_{L^2(\Omega)^d}$.}
\begin{equation*}
    T_k = DR(u_k, p_k)^*DR(u_k, p_k).
\end{equation*}

\paragraph{Neural Network Discretization}
We choose neural networks $u_\theta$ and $p_\psi$ as an ansatz for the velocity and the pressure and denote the PINN loss by
\begin{equation*}
    L(\theta, \psi)
    =
    \frac12 \| -\Delta u_\theta + (u_\theta\cdot\nabla)u_\theta + \nabla p_\psi - f \|^2_{L^2(\Omega)^d}.
\end{equation*}
Next, we discretize $DR(u_\theta, p_\psi)^*DR(u_\theta, p_\psi)$ in the tangent space of the neural network ansatz following the abstract framework. This yields a block-matrix 
\begin{align*}
    G = G(\theta,\psi)
    =
    \begin{pmatrix}
        A & B \\
        B^T & C
    \end{pmatrix}.
\end{align*}
Setting $\xi_i = \Delta \partial_{\theta_i}u_\theta + (u_\theta\cdot\nabla)\partial_{\theta_i}u_\theta + (\partial_{\theta_i}u_\theta \cdot \nabla)u_\theta$ as an abbreviation, we have
\begin{equation}\label{eq:gnng_A_and_B}
    A_{ij}
    =
    (\xi_i, \xi_j)_{L^2(\Omega)^d}, \quad B_{ij} = (\xi_i, \nabla \partial_{\psi_j}p_\psi)_{L^2(\Omega)^d}
\end{equation}
and 
\begin{equation}\label{eq:gnng_C}
    C_{ij} = (\nabla \partial_{\psi_i}p_\psi,\nabla \partial_{\psi_j}p_\psi)_{L^2(\Omega)^d}.
\end{equation}
Using the above computations, the algorithm in parameter space becomes
\begin{align*}
    \begin{pmatrix}
        \theta_{k+1} \\ \psi_{k+1}
    \end{pmatrix}
    =
    \begin{pmatrix}
        \theta_{k} \\ \psi_{k}
    \end{pmatrix}
    -
    \eta_kG(\theta_k,\psi_k)^\dagger\nabla L(\theta_k,\psi_k).
\end{align*}

\paragraph{Correspondence to Gauss-Newton in Parameter Space}
We show that Gauss-Newton in parameter space and function space coincide, given a suitable integral discretization. For quadrature points $(x_i)_{i=1,\dots,N}$ in $\Omega$ we define the discrete residual $r\colon(\theta,\psi)\to\mathbb R^N$ to be
\begin{align*}
    \frac{|\Omega|}{\sqrt N}
    \begin{pmatrix}
        (-\Delta u_\theta + (u_\theta\cdot\nabla)u_\theta + \nabla p_\psi - f)(x_1) \\
        \vdots \\
        (-\Delta u_\theta + (u_\theta\cdot\nabla)u_\theta + \nabla p_\psi - f)(x_N)
    \end{pmatrix}.
\end{align*}
The discretized PINN formulation of~\eqref{eq:navier_stokes} reads
\begin{equation}\label{eq:navier_stokes_pinn_discrete}
    \min L(\theta,\psi) = \frac12\|r(\theta,\psi)\|^2_{l^2}.
\end{equation}
It is straight-forward to see that applying Gauss-Newton to~\eqref{eq:navier_stokes_pinn_discrete} yields the matrix $G = G(\theta,\psi)$ -- if in the discretization of the integrals in the matrices~\eqref{eq:gnng_A_and_B} and~\eqref{eq:gnng_C} the same quadrature points as in the definition of $r$ are being used. 
We provide the details in Appendix~\ref{appendix:details_gauss_newton}. 

\begin{remark}
    The correspondence between Gauss-Newton in function space and its counterpart in parameter space holds for general nonlinear least-squares problems, not only for the Navier-Stokes equations, see Appendix~\ref{appendix:details_gauss_newton}. 
\end{remark}

\subsection{Scalability}\label{sec:scalability}
Function-space inspired algorithms like those discussed in the previous section yield \rebuttal{require the solution of a large system of linear equations of the size of the parameters $d_\Theta$, which is intractable in high parameter dimension $d_\Theta$.}
Since this is a critical aspect we briefly review matrix-free second-order optimization as discussed in~\cite{schraudolph2002fast}, and the K-FAC approach~\cite{martens2015optimizing} that allows to efficiently compute approximate inverses of $G(\theta_k)$.

\paragraph{Matrix-Free Second-Order Optimization}
In certain situations, the matrix $G(\theta)$ allows for the computation of matrix-vector products $v\mapsto G(\theta)v$ without the need for matrix assembly and storage at a comparable computational cost to the gradient $\nabla L(\theta)$. 
With matrix-vector products available one can resort to iterative linear solvers that only require matrix-vector products, such as the CG or GMRES methods~\cite{trefethen2022numerical}. For example, matrix-vector products are available for PINN-type loss functions of linear PDEs that are optimized using Newton's method in function space. In this setting, the resulting optimization in parameter space is the Gauss-Newton method, see also Remark~\ref{remark:engd_is_gn_sometimes} and Remark~\ref{remark:newton_is_gauss_newton_linear_pde}. More precisely, it holds $G(\theta) = J^T(\theta)\cdot J(\theta)$, where $J(\theta)$ is the Jacobian of a suitably scaled residual $r$. Then, Jacobian-vector products and vector-Jacobian products can be efficiently computed using automatic differentiation, relying on a combination of forward and backward modes. The matrix-vector product $G(\theta)\cdot v$ for a given vector $v\in\mathbb R^p$ can be computed as
\begin{equation*}
    w = J(\theta)v, \quad G(\theta)v = (w^T\cdot J(\theta))^T.
\end{equation*}
Further details on matrix-free optimization methods can be found in~\cite{schraudolph2002fast} and a successful application is demonstrated in~\cite{zeng2022competitive}. 

\paragraph{K-FAC}
Kronecker-Factorized Approximate Curvature (K-FAC) is used to approximate the Fisher information matrix or Gauß-Newton matrix~\cite{martens2015optimizing, eschenhagen2024kronecker}. 
The approach approximates the matrix $G(\theta)$ as the Kronecker product of much smaller matrices $G(\theta) \approx A(\theta) \otimes B(\theta)$, which by the properties of the Kronecker product yields $G(\theta)^{-1} \approx A(\theta)^{-1} \otimes B(\theta)^{-1}$.
The Kronecker structure stems from the linear layers in the neural network ansatz making the concrete K-FAC approximation dependent on the structure of the ansatz set. K-FAC is the state-of-the-art method for optimization in neural network based variational Monte Carlo methods~\cite{pfau2020ab, li2023forward, scherbela2023variational}.

\section{Conclusion \rebuttal{and Outlook}}\label{sec:conclusion}%
\rebuttal{We consider a variety of problems in scientific machine learning for which optimization is arguably the biggest challenge and no principled way or commonly accepted best-practices for optimization exists.}
We provide a principled way to transfer infinite dimensional optimization algorithms to nonlinear neural network ansatz classes which follows the paradigm of \emph{first optimize, then discretize}.
\def\arraystretch{1.1}
\begin{table}
\begin{center}
\begin{tabular}{|c|c|c|}
    \hline
    Function Space & Parameter Space & Name  \\
    \hline
    Gradient Descent & NGD & NGD  \\
    Newton & GGN & ENGD  \\
    Lagrange-Newton & CGD & CPINNs \\
    Gauss-Newton & GN & GNNGD  \\
    \hline
\end{tabular}
\caption{Translation of optimization algorithms; here NGD stands for natural gradient descent, GN for Gauss-Newton, GGN for generalized Gauss-Newton, ENGD for energy natural gradient descent~\cite{pmlr-v202-muller23b}, GNNGD for Gauss-Newton natural gradient descent~\cite{anonymous2024}, and CPINNs for competitive PINNs~\cite{schafer2019competitive}. }\label{table:translation_algorithms}
\end{center}
\end{table}
Here, it is the idea to first choose an algorithm in function space that is well aligned with the problem and then to discretize is using neural networks. 
We show that this approach offers a unified view on many state-of-the-art optimization routines currently employed in SciML, see Table~\ref{table:translation_algorithms}. 
\rebuttal{
This leads us to the following conclusions:
\begin{itemize}
    \item Function-space inspired optimization is currently underdeveloped in the field of scientific machine learning. 
    \item The Function-space perspective yields principled way to design problem-specific optimization algorithms in SciML.
    This has the potential to greatly improve the performance on many current SciML tasks. 
\end{itemize}
}
\rebuttal{
Based on our discussion, we propose the following program for the development of efficient function space-algorithms in scientific machine learning: 
\begin{itemize}
    \item Design function-space inspired methods for more SciML problems, in particular this requrires understanding what an appropriate function-space algorithm is for a given problem at hand. 
    \item Provide fast implementations of function-space inspired algorithms. Note that the methods developed in the other contexts can usually not be applied as the function-space geometries in SciML often incorporate PDE specific terms. 
\end{itemize}
}

\section*{Impact Statement}
This paper presents work whose goal is to advance the field of Machine Learning. There are many potential societal consequences of our work, none which we feel must be specifically highlighted here. 

\section*{Acknowledgements}
JM acknowledges funding by the Deutsche Forschungsgemeinschaft (DFG, German Research Foundation) under the project number 442047500 through the Collaborative Research Center \emph{Sparsity and Singular Structures} (SFB 1481). 

\bibliography{example_paper}
\bibliographystyle{icml2024}

\newpage
\appendix
\onecolumn

\section{Details Abstract Framework}
For now, we will work under the assumption that the map $T_u$ is symmetric and positive semi-definite. 
We denote the inner product induced by $T$ by 
\[ [w_1,w_2]_{T_u} \coloneqq \langle T_u w_1, w_2 \rangle \]
and the generalized norm by $\lVert w \rVert_{T_{u}}^2 = [w, w]_{T_{u}}$.

\subsection{Derivation of Normal Equations of~\eqref{eq:fitting_the_update_direction}}\label{appendix:proof_best_fit}
\begin{lemma}
    We consider a differentiable parametrization $P\colon\Theta\to\mathcal H$ and $d_k = T_{u_{\theta_k}}^{-1}(DE(u_{\theta_k}))\in \mathcal H$, where we 
    assume $T_{u_{\theta_k}}\colon\mathcal H\to\mathcal H^\ast$ to be symmetric, linear and bounded. Then it holds that 
    \begin{equation}\label{eq:appendix_best_fit}
        w_k\in\arg\min \frac12\lVert DP(\theta_k)w- d_k \rVert_{T_{u_{\theta_k}}}^2,
    \end{equation}
    if and only if 
    \begin{align*}
        G(\theta_k) w_k = - \nabla L(\theta_k),
    \end{align*}
    where 
    \begin{align*}
        G(\theta)_{ij} = \langle T_{u_{\theta_k}}\partial_{\theta_i}u_{\theta_k}, \partial_{\theta_j}u_\theta \rangle.
    \end{align*}
\end{lemma}
\begin{proof}
    The right-hand side in~\eqref{eq:appendix_best_fit} is up to the constant term $\frac12\lVert d_k \rVert_{T_{u_{\theta_k}}}^2$ 
    given by 
    \begin{align}\label{eq:aux_loss}
        \ell(w) \coloneqq \frac12 \lVert DP(\theta_k) w\rVert _{T_{u_{\theta_k}}}^2  -  [DP(\theta_k) w, d_k]_{T_{u_{\theta_k}}}. 
    \end{align}
    Using $DP(\theta)w = \sum_i w_i \partial_{\theta_i}u_\theta$ and the definition of  $[\cdot, \cdot]_{T_u}$, the first term of~\eqref{eq:aux_loss} takes the form 
    \begin{align*}
    \frac12\sum_{i,j} w_i w_j \langle T_{u_{\theta_k}}\partial_{\theta_i}u_{\theta_k}, \partial_{\theta_j}u_{\theta_k} \rangle  = \frac12 w^\top G(\theta_k) w,
    \end{align*}
    where 
        $G(\theta)_{ij} = \langle T_{u_{\theta_k}}\partial_{\theta_i}u_{\theta_k}, \partial_{\theta_j}u_\theta \rangle$. 
    The second term in~\eqref{eq:aux_loss} amounts to 
    \begin{align*}
        -\sum_{i}\langle w_i\partial_{\theta_i}u_{\theta_k} , T_{u_{\theta_k}}d_k\rangle & = \sum_{i}\langle\partial_{\theta_i}u_{\theta_k}, DE(u_{\theta_k})\rangle 
        = w^\top \nabla L(\theta_k),
    \end{align*}
    where we used $d_k = - T_{u_{\theta_k}}^{-1} (DE(u_{\theta_k}))$ and the chain rule $\partial_{\theta_i}L(\theta) = DE(u_{\theta}) \partial_{\theta_i} u_\theta$. 
    Overall, this yields 
    \begin{align*}
        \ell(w) = \frac12w^\top G(\theta_k) w + w^\top \nabla L(\theta_k),
    \end{align*}
    where the optimizers of this function are characterized by 
    \begin{align*}
        0 = \nabla \ell(w) = G(\theta_k)w + \nabla L(\theta_k).
    \end{align*}
\end{proof}


\subsection{Proof of Theorem~\ref{thm:projection_theorem}}\label{sec:appendix_update_direction}

\projectionthm*
\begin{proof}
This is a direct consequence of~\eqref{eq:appendix_best_fit} as the best approximation is given by the orthogonal projection. 
\end{proof}

\subsection{Extension to Non-Symmetric and Indefinite $T_u$}\label{sec:extension_to_non_symmetric}

Recall that we consider an iterative algorithm in a Hilbert space $\mathcal H$ of the form  
\begin{equation*}
    u_{k+1} = u_k + \eta_k d_k, \quad \text{with }d_k = T_{u_k}^{-1}(DE(u_k)).
\end{equation*}
Here, $T_{u_k}:\mathcal H \to \mathcal H^*$ is a continuous, linear, and bijective map, $E:\mathcal H\to \mathbb R$ is a function whose extrema or saddle points we aim to find. The update direction $d_k$ satsifies the equation
\begin{equation*}
    \langle T_{u_k} d_k, w \rangle
    =
    \langle DE(u_k), w \rangle 
    \quad
    \text{for all }
    w\in\mathcal H.
\end{equation*}
Now we again introduce a neural network discretization of this algorithm and assume that we are at step $k$ with parameters $\theta_k$ and the corresponding function $u_{\theta_k}$. The discretization in tangent space and the solution of equation~\eqref{eq:gramian_update_direction} is nothing but replacing the space of test functions by the tangent space of the neural network ansatz at $\theta_k$, i.e., with $w_k = G(\theta_k)^\dagger\nabla L(\theta_k)$ and $d_{\theta_k} = DP(\theta_k)w_k$ we have
\begin{equation*}
    \langle T_{u_{\theta_k}} d_{\theta_k}, w \rangle 
    =
    \langle DE(u_{\theta_k}), w \rangle
    \quad
    \text{for all }
    w\in T_{u_{\theta_k}}\mathcal M.   
\end{equation*}
Here $T_{u_{\theta_k}}\mathcal M$ is the tangent space of the neural network ansatz at $u_{\theta_k}$, i.e., $\operatorname{span}\{ \partial_{\theta_1}u_{\theta_k}, \dots \partial_{\theta_p}u_{\theta_k} \}$. For general invertible, but possibly non-symmetric and indefinite $T_{u_{\theta_k}}$, we are interested in quasi-best approximation results of the form
\begin{equation}\label{eq:generalized_cea}
    \| d_k - d_{\theta_k} \|_{\mathcal H}
    \leq 
    C\cdot
    \inf_{d\in T_{u_{\theta_k}}\mathcal M}\| d_k - d \|_{\mathcal H}.
\end{equation}
Such an estimate satisfies the best approximation property of an orthogonal projection up to a constant $C$\footnote{This constant should ideally not depend on critical parameters.} and thus guarantees that the function space update directions $d_k$ are closely matched. Results like~\eqref{eq:generalized_cea} are well-known in the finite element literature and hold if certain inf-sup conditions are satisfied both on the continuous and the discrete level. We refer to~\cite{xu2003some, boffi2013mixed} for an introduction. Note that the verification depends on the properties on the infinite-dimensional level \emph{and} the concrete structure of the discrete ansatz spaces which is not the case for symmetric and positive definite operators.

\section{Discretization of Saddle Point Problems}\label{appendix:saddle_point_discretization}
We recall our notation for saddle point problems. Given a map 
\begin{equation*}
    \mathcal L:\mathcal H \times \mathcal V \to \mathbb R, \quad (u,v)\mapsto \mathcal L(u,v)
\end{equation*}
we are looking for a saddle point $(u^*,v^*)$. For their solution we employ an iterative algorithm of the form
\begin{equation*}
    (u_{k+1}, v_{k+1}) = (u_{k}, v_{k}) + (d_{u_{k}}, d_{v_{k}}).
\end{equation*}
We assume that $\mathcal L$ is Fr\'echet differentiable and the update direction is given by 
\begin{equation*}
    (d_{u_{k}}, d_{v_{k}}) = T_{u_k,v_k}^{-1}(D\mathcal L(u,v)).
\end{equation*}
Here, 
\begin{equation*}
 T_{u_k,v_k}:\mathcal H\times\mathcal V \to \mathcal H^*\times\mathcal V^*   
\end{equation*}
is a bounded, linear, and invertible map. Recall our notation for the competitive loss $L$
\begin{equation*}
    L:\Theta \times \Psi \to \mathbb R, \quad L(\theta, \psi) = \mathcal L(u_\theta, v_\psi),
\end{equation*}
where $u_\theta$ and $v_\psi$ are two neural networks with parameter spaces $\Theta$ and $\Psi$, respectively. We then employ
\begin{equation*}
    \{(\partial_{\theta_i}u_\theta,0)\}_{i=1,\dots,p_\Theta}
    \quad \textrm{and}\quad
    \{(0,\partial_{\psi_i}v_\psi)\}_{i=1,\dots,p_\Psi} 
\end{equation*}
for the discretization in the neural network's tangent space. 
To discretize the linear map $T = T_{u_\theta,v_\psi}$, note that we can write it in block structure
\begin{align*}
    T 
    =
    \begin{pmatrix}
        T^1 & T^2 \\
        T^3 & T^4
    \end{pmatrix}
\end{align*}
with $T_1:\mathcal H\to\mathcal H^*$, $T_2:\mathcal V \to \mathcal H^*$, $T_3:\mathcal H \to \mathcal V^*$, and $T_4: \mathcal V \to \mathcal V^*$. The corresponding matrix $G=G(\theta, \psi)$ inherits this block structure
\begin{align*}
    G 
    =
    \begin{pmatrix}
        G^1 & G^2 \\
        G^3 & G^4
    \end{pmatrix}
\end{align*}
and it holds 
\begin{equation*}
    G^1_{ij} = \langle T_1 \partial_{\theta_i}u_\theta, \partial_{\theta_j}u_\theta \rangle_{\mathcal H}, 
    \quad 
    G^2_{ij} = \langle T_2 \partial_{\psi_j}v_\psi, \partial_{\theta_i}u_\theta \rangle_{\mathcal H}
\end{equation*}
and
\begin{equation*}
    G^3_{ij} = \langle T_3 \partial_{\theta_j}u_\theta, \partial_{\psi_i}v_\psi \rangle_{\mathcal V}, 
    \quad 
    G^2_{ij} = \langle T_4 \partial_{\psi_i}v_\psi, \partial_{\psi_j}v_\psi \rangle_{\mathcal V}.
\end{equation*}
For the derivative of $\mathcal L$ note that the chain rule implies
\begin{equation*}
    (D_u\mathcal L(u_\theta,v_\psi)(\partial_{\theta_i}u_\theta), D_v\mathcal L(u_\theta,v_\psi)(\partial_{\psi_j}v_\psi)) 
    =
    (\nabla_\theta L(\theta,\psi)_i, \nabla_\psi L(\theta,\psi)_j),
\end{equation*}
i.e., corresponds to the gradient of the function $L$.
The algorithm in parameter space with the additional introduction of a step size $\eta_k> 0$ reads
\begin{equation*}
    (\theta_{k+1}, \psi_{k+1})
    =
    (\theta_k, \psi_k)
    -
    \eta_k
    G_k^\dagger \nabla L(\theta_k,\psi_k),
\end{equation*}
A typical damping strategy consists of adding an $\epsilon$-scaled identity to $G$ and reads as 
\begin{equation*}
    (\theta_{k+1}, \psi_{k+1})
    =
    (\theta_k, \psi_k)
    -
    \eta_k
    (G_k + \epsilon_k\operatorname{Id})^\dagger \nabla L(\theta_k,\psi_k).
\end{equation*}

\begin{remark}[Interpretation of Update Direction]
    A similar projection result to the one provided for minimization problems can obtained for discretized saddle point problems \emph{under appropriate assumtions on $T$}, see Appendix~\ref{sec:extension_to_non_symmetric}.
\end{remark}

\section{Extended Examples}\label{sec:appendix_proofs}
This Appendix provides detailed proofs for omitted details in Section~\ref{sec:concrete_examples}.
\subsection{Details for Section~\ref{sec:lagrange_newton} on Lagrange-Newton Methods}\label{sec:appendix_proofs_lagrange_newton}

Recall that we aim to solve
\begin{align}\label{appendix:solution_constrained_problem}
\begin{split}
    \min_{u\in H^2(\Omega)} J(u) = 0 \quad \text{s.t. }
    \begin{cases}
        \Delta u + f & = 0 \quad \text{in }\Omega,
    \\
    u - g & = 0 \quad \text{on }\partial\Omega.
    \end{cases}
\end{split} 
\end{align}
Given $\Omega\subset\mathbb R^d$ open and bounded with a $C^{1,1}$ boundary, $f\in L^2(\Omega)$, and $g\in H^{3/2}(\partial\Omega)$. Note that elliptic regularity theory~\cite{grisvard2011elliptic} guarantees the existence of a solution $u^*\in H^2(\Omega)$ to this problem.

\paragraph{Equivalence of Saddle Points and Critical Points of the Lagrangian}
We now considering the Lagrangian of the constrained minimization problem which is given by
\begin{equation}
    \mathcal L\colon H^2(\Omega)\times L^2(\Omega) \times L^2(\partial\Omega) \to \mathbb R,
    \quad
    \mathcal L(u,\lambda, \mu) 
    =
    (\lambda , f + \Delta u)_\Omega
    +
    (\mu, u - g)_{\partial\Omega}.
\end{equation}
As $\mathcal L$ is a sum of continuous bilinear forms, it is Fr\'echet differentiable with derivative
\begin{equation}\label{eq:derivative_L_Lagrange_Newton}
    D\mathcal L(u,\lambda, \mu)((\delta_u, \delta_\lambda, \delta_\mu))
    =
    ((\lambda, \Delta \delta_u)_\Omega + (\mu,\delta_u)_{\partial\Omega}, (\delta_\lambda, f + \Delta u)_\Omega, (\delta_\mu, u-g)_{\partial\Omega}).
\end{equation}
Recall that a saddle point -- or in this case equivalently a Nash equilibrium -- of $\mathcal L$ is a triplet $(u^*, \lambda^*, \mu^*)$ that satisfies 
\begin{equation}\label{appendix:def_saddle_point}
    \mathcal L(u^*, \lambda, \mu) \leq \mathcal L(u^*, \lambda^*, \mu^*) \leq \mathcal L(u, \lambda^*, \mu^*), \quad \forall u, \lambda, \mu \in H^2(\Omega)\times L^2(\Omega) \times L^2(\partial\Omega).
\end{equation}
It is well-known, see for instance~\cite{zeidler2012applied} Theorem 2.F, that a saddle point $(u^*, \lambda^*, \mu^*)$ satisfies the minimax problem
\begin{equation*}
    \mathcal L(u^*, \lambda^*, \mu^*)
    =
    \max_{\lambda, \mu}\min_{u} \mathcal L(u,\lambda, \mu)
    =
    \min_{u}\max_{\lambda, \mu} \mathcal L(u,\lambda, \mu)
    =
    \min_{u}\max_{\lambda, \mu}
    \Big[ 
    (\lambda, f + \Delta u)_\Omega 
    +
    (\mu, u-g)_{\partial\Omega}
    \Big].
\end{equation*}
The minimax formulation above is the CPINN formulation proposed in~\cite{zeng2022competitive} applied to Poisson's equation.

\begin{lemma}
    The unique saddle point of $\mathcal L$ is $(u^*,0,0)$, where $u^*\in H^2(\Omega)$ denotes the solution to~\eqref{appendix:solution_constrained_problem}. Furthermore, the triplet $(u^*, 0, 0)$ is the unique zero of $D\mathcal L$.
\end{lemma}
\begin{proof}
    It is easy to see that the triplet $(u^*, 0, 0)$ satisfies the defining inequalities~\eqref{appendix:def_saddle_point} of a saddle point as all expressions evaluate to zero. Given another triplet $(u^{**}, \lambda^{**}, \mu^{**})$ that satisfies~\ref{appendix:def_saddle_point} we realize that the condition
    \begin{equation*}
        \mathcal L(u^{**},\lambda, \mu)
        \leq
        \mathcal L(u^{**}, \lambda^{**}, \mu^{**}), \quad \textrm{for all }\lambda\in L^2(\Omega),\mu\in L^2(\partial\Omega)
    \end{equation*}
    can only hold with a finite value for $\mathcal L(u^{**}, \lambda^{**}, \mu^{**})$ if $u^{**}=u^{*}$. Moreover 
    \begin{equation*}
        \mathcal L(u^{*}, \lambda^{**}, \mu^{**})
        \leq
        \mathcal L(u,\lambda^{**}, \mu^{**}), \quad \textrm{for all }u\in H^2(\Omega)
    \end{equation*}
    can only hold with a finite value for $\mathcal L(u^{*}, \lambda^{**}, \mu^{**})$ if $\lambda^{**}=\mu^{**}=0$. Hence $(u^{**}, \lambda^{**}, \mu^{**}) = (u^*, 0, 0)$.

    Investigating the derivative of the Lagrangian~\eqref{eq:derivative_L_Lagrange_Newton}, we see that $(u^*, 0, 0)$ is a critical point. Furthermore, for any other critical point $(u^{**}, \lambda^{**}, \mu^{**})$ the last two equations of~\eqref{eq:derivative_L_Lagrange_Newton} imply that $u^{**}= u^*$. The first equation reads 
    \begin{equation*}
        (\lambda, \Delta \delta_u)_\Omega + (\mu, \delta_u)_{\partial\Omega} = 0 \quad \textrm{for all }\delta_u\in H^2(\Omega).
    \end{equation*}
    We assume first that $\mu\in H^{3/2}(\partial\Omega)$ and consider the auxiliary problem of finding $\delta_u^*\in H^2(\Omega)$ that satisfies
    \begin{align*}
        \begin{split}
            \Delta \delta_u^* &= \lambda \quad \textrm{in }\Omega,
            \\
            \delta_u^* &= \mu \quad \textrm{on }\partial\Omega.
        \end{split}
    \end{align*}
    Testing with this $\delta_u^*$ yields 
    \begin{equation*}
        \| \lambda \|_{L^2(\Omega)}^2 + \|\mu\|_{L^2(\partial\Omega)}^2 = 0 
    \end{equation*}
    and thus the desired fact $\lambda=\mu=0$. If $\mu$ is merely $L^2(\partial\Omega)$ we use the density of $H^{3/2}(\partial\Omega)\subset L^2(\partial\Omega)$ and replace $\mu$ in the auxiliary equation by $\mu_\epsilon\in H^{3/2}(\partial\Omega)$ such that $(\mu,\mu_\epsilon)_{L^2(\partial\Omega)}>0$. This then implies again $\lambda = 0$ and then also $\mu=0$.
\end{proof}

\paragraph{Correspondence of Lagrange-Newton and Competitive Gradient Descent}
We provide the missing computations to verify the correspondence claimed in the main text. Recall that 
\begin{equation*}
    D^2\mathcal L(u_0,\lambda_0, \mu_0)((\delta_u, \delta_\lambda, \delta_\mu),(\bar\delta_u, \bar\delta_\lambda, \bar\delta_\mu))
\end{equation*}
is given by
\begin{align*}
\begin{pmatrix}
    \delta_u & \delta_\lambda & \delta_\mu
\end{pmatrix}
    \begin{pmatrix}
         0 & (\cdot, \Delta \cdot)_\Omega & (\cdot, \cdot)_{\partial\Omega}
         \\
         (\cdot, \Delta\cdot)_{\Omega} & 0 & 0
         \\
         (\cdot, \cdot)_{\partial\Omega} & 0 & 0
    \end{pmatrix}
    \begin{pmatrix}
        \bar\delta_u
        \\
        \bar\delta_\lambda
        \\
        \bar\delta_\mu
    \end{pmatrix}.
\end{align*}
Discretizing this matrix in the tangent space of the neural network ansatz, i.e., using the functions $\partial_{\theta_i}u_\theta, \partial_{\psi_i}\lambda_\psi, \partial_{\xi_i}\mu_\xi$ yields a block matrix of the form
\begin{align}
    G(\theta, \psi, \xi) = 
    \begin{pmatrix}
        0   & A & B \\
        A^T & 0 & 0 \\
        B^T & 0 & 0
    \end{pmatrix},
\end{align}
where
\begin{equation*}
    A_{ij} = (\partial_{\psi_j}\lambda_\psi, \Delta \partial_{\theta_i}u_\theta)_\Omega, 
    \quad
    B = (\partial_{\xi_j}\mu_\xi, \partial_{\theta_i}u_\theta).
\end{equation*}

The matrix employed in competitive gradient descent for this problem is
\begin{align*}
    \begin{pmatrix}
    \operatorname{Id} & \eta D^2_{\theta,\psi}L & \eta D^2_{\theta,\xi}L 
    \\
    \eta D^2_{\psi\theta}L & \operatorname{Id} & \eta D^2_{\psi,\xi}L
    \\
    \eta D^2_{\xi,\theta}L & \eta D^2_{\xi,\psi}L & \operatorname{Id}
    \end{pmatrix}.
\end{align*}
For the reader's convenience we recall the loss function 
\begin{equation*}
    L(\theta, \psi, \xi)
    =
    (\lambda_\psi, \Delta u_\theta + f)_\Omega + (\mu_\xi, u_\theta - g)_{\partial\Omega}.
\end{equation*}
We compute the partial derivatives of $L$
\begin{align*}
    \partial_{\theta_i}L(\theta, \psi, \xi) 
    &=
    (\lambda_\psi, \Delta \partial_{\theta_i}u_\theta)_{\Omega} + (\mu_\xi, \partial_{\theta_i}u_\theta)_{\partial\Omega},
    \\
    \partial_{\psi_j}\partial_{\theta_i}L(\theta, \psi, \xi) 
    &=
    (\partial_{\psi_j}\lambda_\psi, \Delta \partial_{\theta_i}u_\theta)_{\Omega},
    \\
    \partial_{\xi_j}\partial_{\theta_i}L(\theta, \psi, \xi) 
    &=
    (\partial_{\xi_j}\mu_\xi, \partial_{\theta_i}u_\theta)_{\partial\Omega}.
\end{align*}
This shows that the first row and column of the CGD matrix and the Lagrange-Newton matrix agree. We proceed to compute
\begin{align*}
    \partial_{\xi_i}L(\theta, \psi, \xi) 
    &=
    (\partial_{\xi_i}\mu_\xi, u_\theta - g)_{\partial\Omega},
    \\
    \partial_{\psi_j}\partial_{\xi_i}L(\theta, \psi, \xi) 
    &=
    0,
\end{align*}
which guarantees the correspondence of the remaining blocks. Note that the derivatives
\begin{equation*}
    \partial_{\theta_i}^2L(\theta, \psi, \xi), \quad \partial_{\psi_i}^2L(\theta, \psi, \xi), \quad 
    \partial_{\theta_i}^2L(\theta, \psi, \xi), \quad 
    \partial_{\xi_i}^2L(\theta, \psi, \xi)
\end{equation*}
do not vanish unless $u_\theta, \lambda_\psi$ and $\mu_\xi$ are linear in $\theta$, $\psi$ and $\xi$, respectively. This shows once more that applying Newton's method to the discrete problem does not reproduce the infinite-dimensional algorithm.

\subsection{Details for Section~\ref{sec:gauss_newton_f_space} on Gauss-Newton Methods}\label{appendix:details_gauss_newton}
In this Section we derive the Gauss-Newton method in function spaces for the solution of nonlinear least-squares problems and discuss the connection to the classical Gauss-Newton method in Euclidean space. 

\paragraph{Derivation of the Gauss-Newton Method}
Assume we are given a Fr\'echet differentiable function $R:\mathcal H \to L^2(\Omega)$ defined on a Hilbert space $\mathcal H$ and we aim to minimize the energy
\begin{equation*}
    E:\mathcal H \to \mathbb R, \quad E(u) = \frac12\|R(u)\|^2_{L^2(\Omega)}.
\end{equation*}
Here we discuss the approach for an abstract space $\mathcal H$. To obtain the formulas for Navier-Stokes as discussed in Section~\ref{sec:gauss_newton_f_space} set
\begin{equation*}
    \mathcal H = H^2(\Omega)^d\times H^1(\Omega). 
\end{equation*}
Given an iterate $u_k\in\mathcal H$, the Gauss-Newton algorithm produces an update $u_{k+1} = u_k + d_k$ by adding the direction $d_k$ which is given via 
\begin{equation*}
    d_k 
    =
    \underset{v\in\mathcal H}{\operatorname{argmin}} \frac12 \| R(u_k) + DR(u_k)v \|^2_{L^2(\Omega)}
    \approx
    \frac12 \| R(u_k + v) \|^2_{L^2(\Omega)}
\end{equation*}
The optimality condition of the quadratic problem that $d_k$ needs to satisfy is given by
\begin{equation*}
    DR(u_k)^*R(u_k) + DR(u_k)^*DR(u_k)d_k = 0.
\end{equation*}
This is an equality in the Hilbert space $\mathcal H$ and $DR(u_k)^*:L^2(\Omega)\to \mathcal H$ denotes the Hilbert space adjoint of the map $DR(u_k)$. Applying the Riesz isomorphism $\mathcal I:\mathcal H\to\mathcal H^*$ to the equation above yields
\begin{equation*}
    DE(u_k) + \mathcal I DR(u_k)^*DR(u_k)d_k = 0
\end{equation*}
as an equality in $\mathcal H^*$ and it shows that for any $u\in\mathcal H$ the map $T_{u}$ for Gauss-Newton's method is given by
\begin{equation*}
    T_{u}:\mathcal H \to \mathcal H^*, \quad T_u = \mathcal I \circ DR(u)^*\circ DR(u).
\end{equation*}
Thus, for a neural network discretization $u_\theta$ the Gramian resulting from $T_{u_\theta}$ is given by
\begin{equation*}
    G(\theta)_{ij} 
    =
    (DR(u_\theta)[\partial_{\theta_i}u_\theta], DR(u_\theta)[\partial_{\theta_j}u_\theta])_{L^2(\Omega)}.
\end{equation*}
Approximating the integrals in the inner product above by Monte Carlo sampling\footnote{Any other quadrature is equally possible.} $x_1,\dots,x_N\in\Omega$ yields 
\begin{equation}\label{appendix:gramian_gauss_newton_fspace}
    G(\theta)_{ij} 
    \approx 
    \frac{|\Omega|}{N}\sum_{k=1}^{N_\Omega}DR(u_\theta)[\partial_{\theta_i}u_\theta](x_k)DR(u_\theta)[\partial_{\theta_j}u_\theta](x_k).
\end{equation}


\paragraph{Correspondence to Gauss-Newton in Parameter Space}
Here we prove the correspondence between Gauss-Newton in function space and parameter space. Consider the residual function $r:\Theta \to \mathbb R^N$
\begin{align*}
    r(\theta) 
    =
    \sqrt{\frac{|\Omega|}{N}}
    \begin{pmatrix}
        R(u_\theta)(x_1)\\
        \vdots\\
        R(u_\theta)(x_N)
    \end{pmatrix}.
\end{align*}
Then we can define a discrete loss function
\begin{equation*}
    L(\theta) 
    =
    \frac{1}{2}\| r(\theta) \|^2_{l^2}
    \approx
    \frac12 \| R(u_\theta) \|^2_{L^2(\Omega)}
    =
    E(u_\theta).
\end{equation*}
Applying the Euclidean version of the Gauss-Newton method to $r$ requires the Jacobian $J$ of $r$, which is given by
\begin{align*}
    J(\theta) 
    =
    \sqrt{\frac{|\Omega|}{N}}
    \begin{pmatrix}
        DR(u_\theta)[\partial_{\theta_1}u_\theta](x_1) & \dots & DR(u_\theta)[\partial_{\theta_p}u_\theta](x_1)\\
        & \vdots & \\
        DR(u_\theta)[\partial_{\theta_1}u_\theta](x_N) & \dots & DR(u_\theta)[\partial_{\theta_p}u_\theta](x_N)
    \end{pmatrix}.
\end{align*}
It is clear that $J^T(\theta)J(\theta)$ is the same as the approximation of the Gramian in equation~\eqref{appendix:gramian_gauss_newton_fspace}, given that the same quadrature points are used.
\begin{remark}[Correspondence of Newton and Gauss-Newton for linear PDEs]\label{remark:newton_is_gauss_newton_linear_pde}
    In the case of a linear PDE, i.e., when $R$ is an affine linear map the previous computations show that Newton's method in function space, as described in Section~\ref{sec:newton_f_space}, also leads to Gauss-Newton's method in parameter space. In this case, $DR(u_k)$ is independent of $u_k$ and agrees with the linear part of $R$.
\end{remark}
\end{document}